\documentclass[sigconf]{acmart}

\AtBeginDocument{%
  }

\copyrightyear{2024}
\acmYear{2024}
\setcopyright{rightsretained}
\acmConference[SETN 2024]{13th Conference on Artificial Intelligence}{September 11--13, 2024}{Piraeus, Greece}
\acmBooktitle{13th Conference on Artificial Intelligence (SETN 2024), September 11--13, 2024, Piraeus, Greece}\acmDOI{10.1145/3688671.3688763}
\acmISBN{979-8-4007-0982-1/24/09}

\usepackage{xspace}
\usepackage{cleveref}
\usepackage{xcolor}
\usepackage{makecell}
\usepackage{subcaption}

\crefname{equation}{Eq.}{Eq.}
\creflabelformat{equation}{#2\textup{#1}#3}
\usepackage{amsfonts}       

\definecolor{darkyellow}{rgb}{0.95, 0.73, 0.24}

\newcommand{\lipo}{\textsc{LIPO}\@\xspace}
\newcommand{\lipop}{\textsc{LIPO\texttt{+}}\@\xspace}
\newcommand{\adalipo}{\textsc{AdaLIPO}\@\xspace}
\newcommand{\adalipop}{\textsc{AdaLIPO\texttt{+}}\@\xspace}
\newcommand{\adalipob}{\textsc{AdaLIPO}\texttt{+}$|$ns\@\xspace}

\newcommand{\R}{\mathbb{R}}

\renewcommand{\P}{\mathbb{P}}

\newcommand{\Z}{\mathbb{Z}}
\newcommand{\cal}[1]{\mathcal{#1}}
\newcommand{\X}{\cal{X}}

\newcommand{\eg}{e.g.\@\xspace}
\newcommand{\ie}{i.e.\@\xspace}
\newcommand{\wrt}{w.r.t.\@\xspace}

\usepackage{hyperref}
\definecolor{mydarkblue}{rgb}{0,0.08,0.45}
\definecolor{mylightpurple}{rgb}{0.26, 0.32, 0.82}
\definecolor{mylightblue}{rgb}{0.26, 0.38, 0.93}
\hypersetup{ %
  colorlinks=true,
  linkcolor=red,
  citecolor=mylightpurple,
  filecolor=mylightpurple,
  urlcolor=mylightpurple,
}

\DeclareCaptionLabelSeparator{point}{.~}
\newcounter{marginNoteCounter}

\newcommand{\etc}   		{etc.\xspace}

\newcommand{\bt}{\color[HTML]{26A96C}}

\begin{document}
\title{\lipop : Frugal Global Optimization for Lipschitz Functions}

\author{Gaëtan Serré \quad Perceval Beja-Battais \quad Sophia Chirrane \\Argyris Kalogeratos \quad Nicolas Vayatis}
\affiliation{%
\institution{École Normale Supérieure Paris-Saclay, Centre Borelli}
  \city{Gif-Sur-Yvette, 91190}
  \country{France}
}
\thanks{E-mail contacts:~name.surname@ens-paris-saclay.fr.%
}



\renewcommand{\shortauthors}{Serre et al.}

\begin{abstract}
In this paper, we propose simple yet effective empirical improvements to the algorithms of the \lipo family, introduced in~\cite{Malherbe2017}, that we call \lipop and \adalipop. We compare our methods to the vanilla versions of the algorithms over standard benchmark functions and show that they converge significantly faster. Finally, we show that the \lipo family is very prone to the curse of dimensionality and tends quickly to Pure Random Search when the dimension increases. We give a proof for this, which is also formalized in Lean mathematical language. Source codes and a demo are provided online.
\end{abstract}



\maketitle

\section{Introduction}
Global optimization methods aim at finding the global maxima of a non-convex function, unknown a priori, over a compact set. This branch of applied mathematics has been extensively studied since it has countless impactful applications. Indeed, optimizing an unknown function is a common problem in many fields such as machine learning, physics, biology, \etc (\eg \cite{Pinter1991, Lee2017}). In this context, only local information about the function is available. Moreover, in many applications, the function is computationally expensive to evaluate. The goal of any global optimization algorithm is to find a precise estimate of the global maximum while minimizing the number of evaluations of the function. For instance, imagine a physical system for which a heavy computer program needs days to simulate possible future trajectories given an initialization and a set of parameters. In such a scenario, without a plan for frugal probing of the function it may be infeasible to optimize the behavior of such a system, with respect to its parameters.

Several approaches have been proposed over the years. Some have theoretical guarantees (\eg \cite{Kirkpatrick1983, Malherbe2017, Rudi2024, Serre2024}) while others are more heuristic (\eg \cite{Hansen1996, Martinezcantin2014, Mirjalili2016, Xue2023}). Most are sequential, \ie they use the information of the previous evaluations to decide where to evaluate the function next, and stochastic, \ie they use a sampling process to explore the space. Recent results show that stochasticity is a key ingredient to achieve good performance (\eg \cite{Zhang2020, Davis2022, Jordan2023}).

In this short paper, we focus on the \lipo family of algorithms, namely the original \lipo algorithm and the adaptive variant \textsc{Ada-LIPO}, both introduced in \cite{Malherbe2017}. This is a family of sequential stochastic optimization algorithms that assume the target is a Lipschitz function (see \cref{eq:Lipschitz} and \cref{fig:Lipschitz}), which they exploit to ensure that algorithms' performance by theoretical guarantees. The Lipschitz assumption is common in many optimization frameworks, as it gives several tools to prove convergence rates of the algorithms, while being general enough to cover a wide range of functions. The contribution of this work is to propose two improved counterpart algorithms, called \lipop and \adalipop, by addressing certain practical limitations of the original versions and by introducing empirical modifications that work better in practice. Both the presented algorithms include a stopping criterion, while the adaptive variant is equipped with a decaying exploration rate. We compare the new algorithms to the vanilla versions on benchmark functions and validate their empirical performance. We show that the \lipo family of methods is very prone to the curse of dimensionality and tends quickly to Pure Random Search when the dimension increases. We give a proof for this and we also formalize it in Lean mathematical language \cite{Moura2021}, which ensures correctness and facilitates future reusability of theoretical results. Last but not least, our companion paper in \cite{Serre2024_IPOL} focuses on the reproducibility of this work, and makes available the pseudocodes, the implementation details, the source code of all the compared algorithms, and an online demo\footnote{\scriptsize Source code:~\url{https://github.com/gaetanserre/LIPO}\\ Demo:~\url{https://ipolcore.ipol.im/demo/clientApp/demo.html?id=469}.}.

\begin{figure*}[t]
  \begin{subfigure}{0.49\textwidth}
	\centering
	\includegraphics[width=0.8\columnwidth]{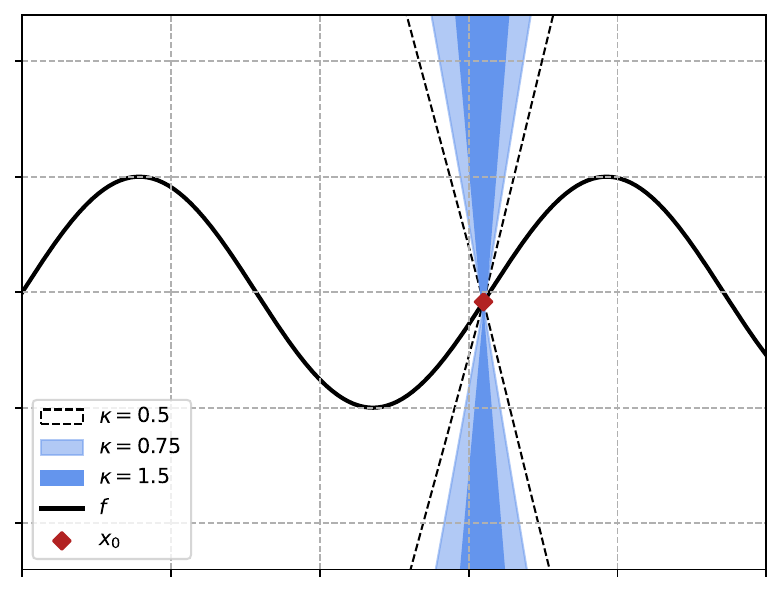}
  \caption{Lipschitz function}
  \label{fig:Lipschitz}
	\end{subfigure}
	\begin{subfigure}{0.49\textwidth}
	\centering
  \includegraphics[width=0.8\columnwidth]{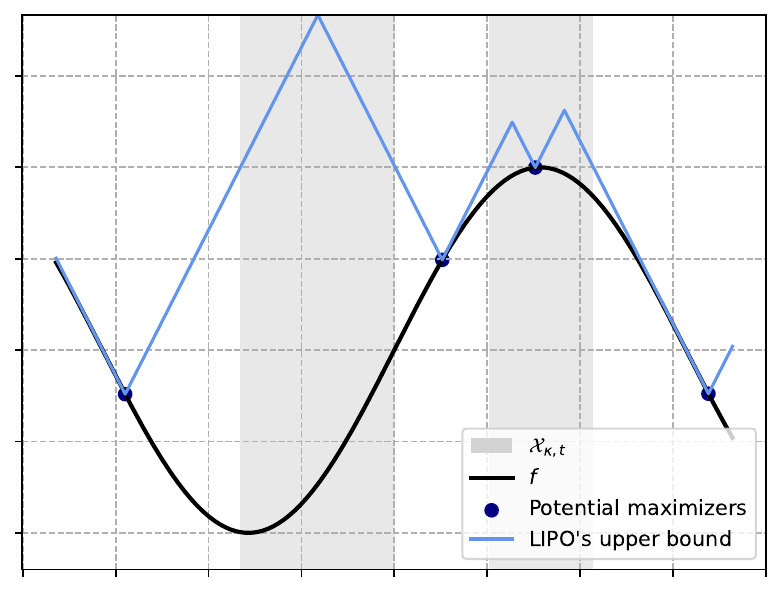}
  \caption{\lipo upper bound}
  \label{fig:ub-lipo}
	\end{subfigure}
	\caption{\textbf{a)}~An example of a Lipschitz function ($x \mapsto 0.5\sin x$). The {\color[HTML]{6495ed} hourglass-shaped double cones} are generated using a slope of $\pm \kappa$. As this function is $0.5$-Lipschitz, the graph of the function is always outside of cones generated by a $\kappa \geq 0.5$. The bigger $\kappa$, the thinner are the cones. \textbf{b)}~Example of an upper bound constructed by \lipo, given the potential maximizers $(X_i)_{1 \leq i < t}$. {\color{gray} $\cal{X}_{\kappa, t}$} is the regions of the domain where the potential maximizers can be found at time $t$, and therefore the regions in which $f$ is worth being probed. The upper bound ({\color[HTML]{6495ed} blue}) is a piecewise linear function with slope coefficient in $\{-\kappa, \kappa\}$ that passes through each potential maximizer.}
\end{figure*}

\section{Lipschitz Optimization}
In this section, we briefly recall the \lipo and \adalipo algorithms. For more details, we refer the reader to \cite{Malherbe2017}. Note that we refer to the \emph{maximization} of a function $f : \X \rightarrow \R$, yet this is only a convention since, due to the ordering property of real numbers, it is equivalent to \emph{minimizing} its negative, \ie $\text{arg\,max}_{x} f(x) = \text{arg\,min}_{x} -\!f(x)$. 

\subsection{LIPO}
Formally, a function $f$ is $\kappa$-Lipschitz in the domain $\X$ when it holds:
\begin{equation}\label{eq:Lipschitz}
\exists \kappa \geq 0, \quad |f(x) -f(x')| \leq \kappa \lVert x-x' \rVert_2 \quad\forall x, x' \in \X \subseteq\R^d.
\end{equation}
where $\kappa$ is the Lipschitz constant, and $d$ denotes the number of dimensions of the space. 
\lipo is a sequential and stochastic method designed to optimize a $\kappa$-Lipschitz function $f$ of known constant $\kappa$, over a domain $\X$ that is compact and convex subset of $\R^d$ with non-empty interior.
Suppose $(X_i)_{1 \leq i < t}$ is the sequence of the $t-1$ potential maximizers found so far.
At each iteration, \lipo samples a candidate point $x\in\X$ uniformly at random, and considers it as a potential maximizer iff:
\begin{equation}\label{eq:lipo}
\max _{1 \leq i < t} f\left(X_i\right) \leq \min _{1 \leq i < t} f \left(X_i\right) + \kappa \left\lVert x - X_i \right\rVert_2,
\end{equation}
By making use of the $\kappa$-Lipschitz property of $f$, \lipo constructs the upper bound of the function at $x$ appearing on the right-hand side of \cref{eq:lipo}, and in this way checks the potential of the the candidate point before evaluating the function at it.
The candidate is considered as a potential maximizer if the upper bound at $X_t$ is greater than the function value at the best maximizer found so far. If the candidate is accepted, \lipo evaluates $f$ at $x$, and the value is stored to compute \cref{eq:lipo} at the next iteration. The fact that $f$ is $\kappa$-Lipschitz ensures that the upper bound is valid and thus no global maximizer is rejected. The algorithm is consistent because the region of potential maximizers tightens around the global maximizers as the number of evaluations increases. \lipo, as well as \adalipo that we will see in the next subsection, are consistent over Lipschitz functions, with a convergence rate of at least $\mathcal{O} \big(t^{-1/d} \big)$ (with high probability), where $t$ is the number of actual function evaluations performed, or in other words the number of potential maximizers found. An example of the upper bound is given in \cref{fig:ub-lipo}. In a nutshell, \lipo maximizes $f$ while minimizing the number of function evaluations by evaluating $f$ only at candidates that are potential maximizers.

Note that, in 
what follows, the number of iterations of \lipo is denoted by $t$, \ie the number of potential maximizers found so far and it should not be confused with the \emph{the number of sampled candidates}. As we assume that $f$ is costly, the number of time it is evaluated is a crucial measure to assess the performance of algorithms, and will be the only one used throughout this paper. The link between the number of samples and the number of potential maximizers is discussed in \cref{sec:limitations-lipo}.

\subsection{Adaptative LIPO}\label{sec:adalipo}
While \lipo is simple and efficient, the Lipschitz constant of $f$ needs to be known. This information is not available in general. As a response to that issue, the \adalipo variant \cite{Malherbe2017} alternates stochastically between two states: a state of exploration, where all candidates get accepted, and a state of exploitation, where an estimation of the Lipschitz constant is used to reject candidates with the \lipo upper bound of \cref{eq:lipo}. The Lipschitz constant at time $t$ is estimated using the follow formula:
$$
  \hat{\kappa}_t \triangleq \inf \left\{\kappa_{i \in \Z}: \max _{i \neq j} \frac{\left|f\left(X_i\right)-f\left(X_j\right)\right|}{\left\lVert X_i-X_j \right\rVert_2} \leq \kappa_i\right\},
$$
where $(\kappa_i)_{i \in \Z}$ can be any real-valued sequence. 
We use the sequence $\kappa_i = (1 + \alpha)^i$, with $\alpha \triangleq 0.01$, following the suggestion of \cite{Malherbe2017}. Thus, one can compute the closed-form expression of $\hat{\kappa}_t$:
\begin{align*}
  \hat{\kappa}_t &= (1 + \alpha)^{i_t}, \\
  &\text{where } i_t = \left \lceil \ln \left( \max_{i \neq j} \frac{|f(X_i) - f(X_j) |}{\|X_i - X_j\|_2} \right) \frac{1}{\ln (1 + \alpha)} \right \rceil.
\end{align*}
This computation ensures that $\kappa$ is not overestimated. Indeed, if there is $i$ such that $\kappa_i \leq \kappa \leq \kappa_{i + 1}$, then, at any time $t$ it holds $\hat{\kappa}_t \leq k_{i + 1}$ (see \cite{Malherbe2017}). At each iteration, \adalipo enters the exploration state with probability $p$, which is a fixed hyperparameter of the algorithm. By accepting any uniformly sampled candidates, the exploration state prevents the estimation of $\kappa$ from being locally biased by the potential maximizers region. This behavior allows the method to be consistent over Lipschitz functions.

\section{Empirical Improvements}

\subsection{Limitations of \lipo}\label{sec:limitations-lipo}
The two main drawbacks of the approach of \lipo are: i)~the calculation of the upper bound becomes computationally more expensive as $t$ grows large, and ii)~the uniform random sampling strategy across all the function domain $\X$. The latter implies that, when the region of potential maximizers gets to be small (see the gray-shaded region $\cal{X}_{\kappa, t}$ in \cref{fig:ub-lipo}), the algorithm needs a long time before it finds a candidate in that region. \cref{fig:time-samples} demonstrates this phenomenon.

To address these issues, we present the \lipop and \adalipop improved versions of the aforementioned algorithms, aiming at a better empirical performance. The first point of improvement is a stopping criterion, motivated by the discussion around \cref{fig:time-samples}, that allows the algorithm to stop when the number of samples required to find a potential maximizer is growing \emph{exponentially}. The second point of improvement concerns \adalipo, where we introduce a \emph{decaying probability} of entering the exploration state. This allows a faster convergence while restricting the approximation of the Lipschitz constant $\kappa$ to the region of interest.

\begin{figure}[tb]
  \centering
  \includegraphics[width=.8\columnwidth]{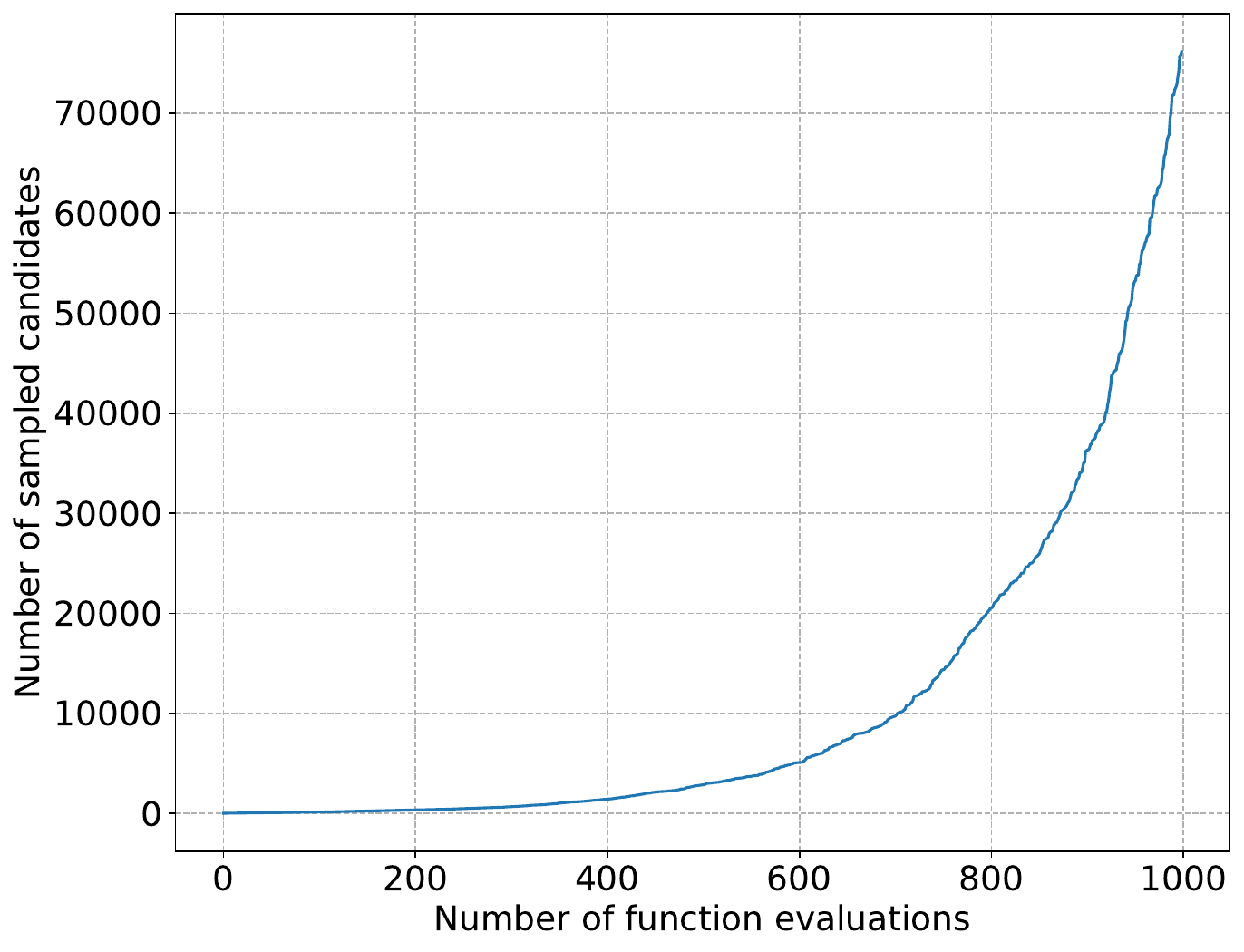}
  \caption{Number of sampled candidates required by \lipo vs number of potential maximizers (\ie evaluations of $f$ on the $x$-axis) for the Rastrigin function. The number of function evaluations corresponds to the number of candidates that satisfy \cref{eq:lipo}. It is evident that more and more candidates need to be sampled for finding the $t$-th potential maximizer as time $t$ passes.}
  \label{fig:time-samples}
\end{figure}

\subsection{Stopping Criterion}\label{sec:stopping-criterion}
As illustrated in \cref{fig:time-samples}, the number of samples required to find a potential maximizer seems to grow exponentially whenever the region of potential maximizers is small: the uniform random sampling has a low probability of finding a candidate in that region. This effect is unavoidable and will eventually happen over time, but the number of previous potential maximizers needed for reaching this state is unknown and depends on both the function and its domain. As it is hard to select a fixed number of evaluations as a stopping criterion, we propose to stop the algorithm whenever the slope of the function represented in \cref{fig:time-samples} exceeds a given threshold $\Delta$. The slope is computed over a window of size $w$. A bigger $\Delta$ value allows the algorithm to run longer, reaching to more precise final approximation.

\subsection{Decaying Exploration Rate}
As explained is \cref{sec:adalipo}, the transition between the exploration and the exploitation state is controlled by a Bernoulli random variable of fixed parameter $p$: $Y \sim \cal{B}(p)$. By intuition and experience, we can assume that exploring a lot at the beginning of the process is a good way to approximate correctly the Lipschitz constant $\kappa$ first, and favor more and more the exploitation state as the number of evaluations increases. We propose to take $p(t) = \min \big(1,\frac{1}{\ln(t)} \big)$, with the convention that $p(1) = 1$. As stated in \cref{sec:adalipo}, the exploration state allows the estimation of $\kappa$ to be unbiased. Our approach does not offer the same guarantee. Indeed, as the iterations increase, the probability of entering the exploration state decreases, and the estimation of $\kappa$ will be restricted to the region of potential maximizers, and hence \adalipop might highly underestimate the global $\kappa$. However, as the relative complement of this region \wrt to $\X$ is ignored by the algorithm, the estimation will be optimal over the current region of potential maximizers. We provide an illustration of the differences between the vanilla and our improved version in \cref{fig:adalipoe-visual}.

\begin{figure*}[t]
  \centering
  \begin{subfigure}{0.25\textwidth}
      \centering
      \includegraphics[width=\textwidth]{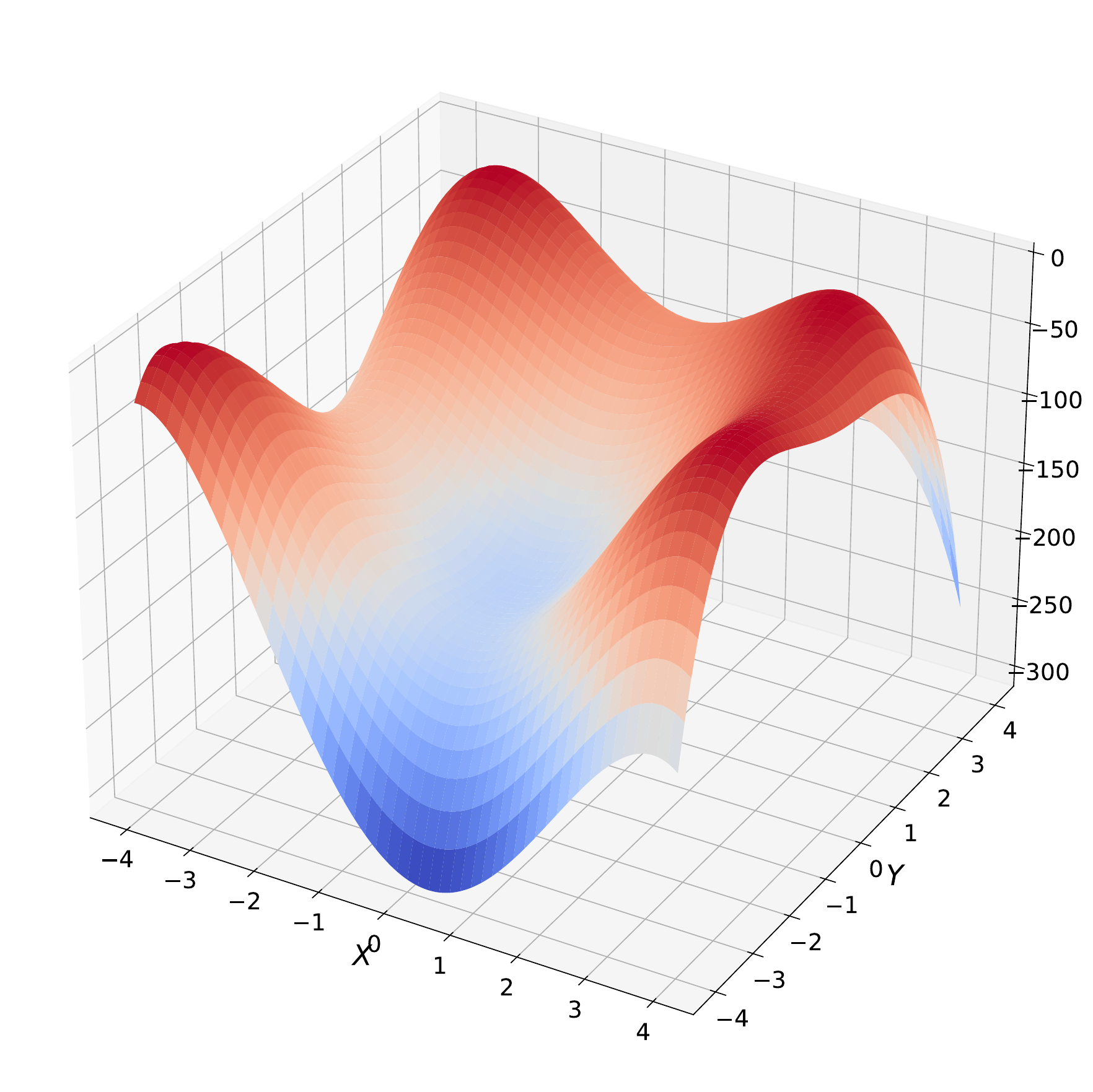}
      \caption{Himmelblau}
  \end{subfigure}
	\hspace{1em}
  \begin{subfigure}{0.25\textwidth}
      \centering
      \includegraphics[width=\textwidth]{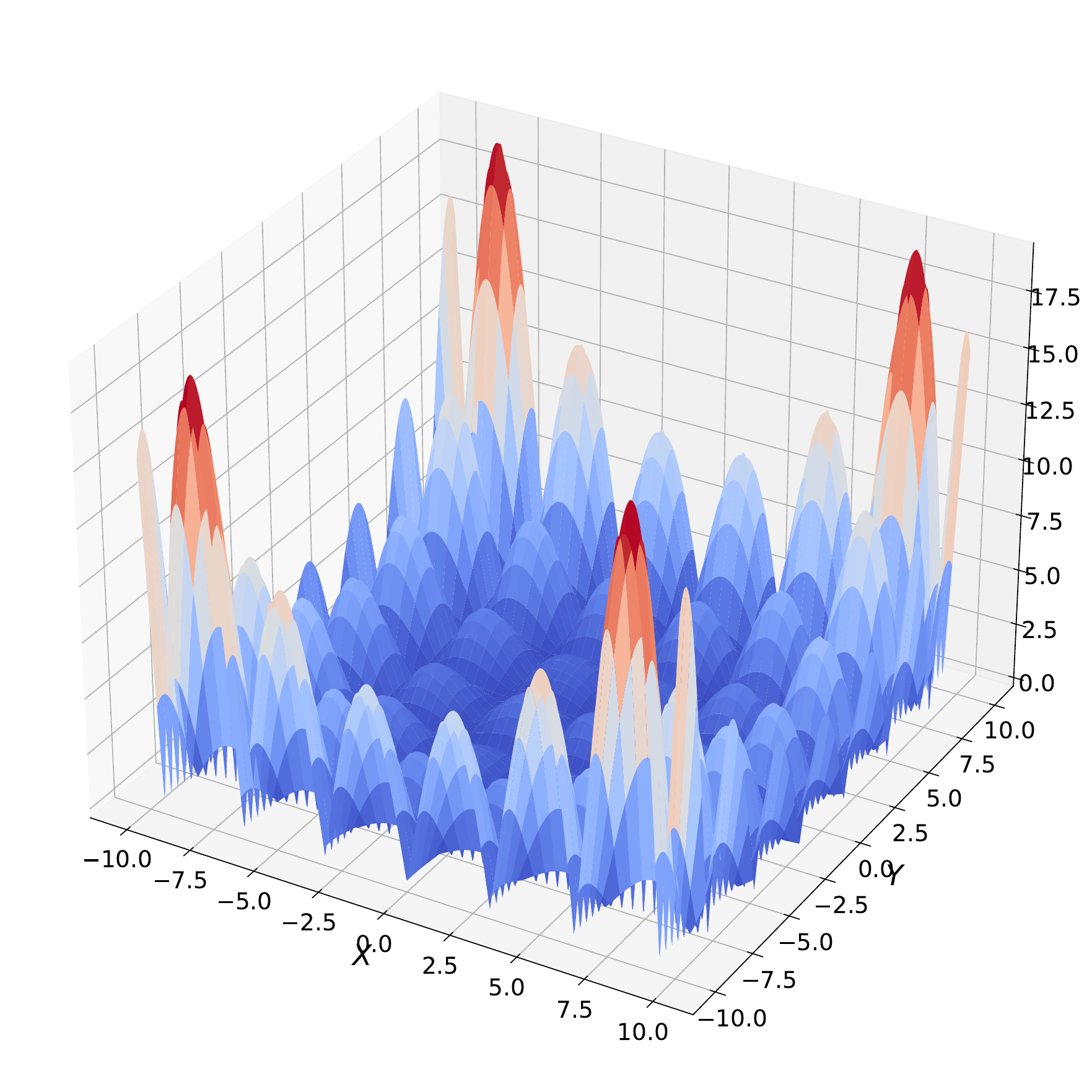}
      \caption{Hölder table}
  \end{subfigure}
  \hspace{1em}
	\begin{subfigure}{0.25\textwidth}
      \centering
      \includegraphics[width=\textwidth]{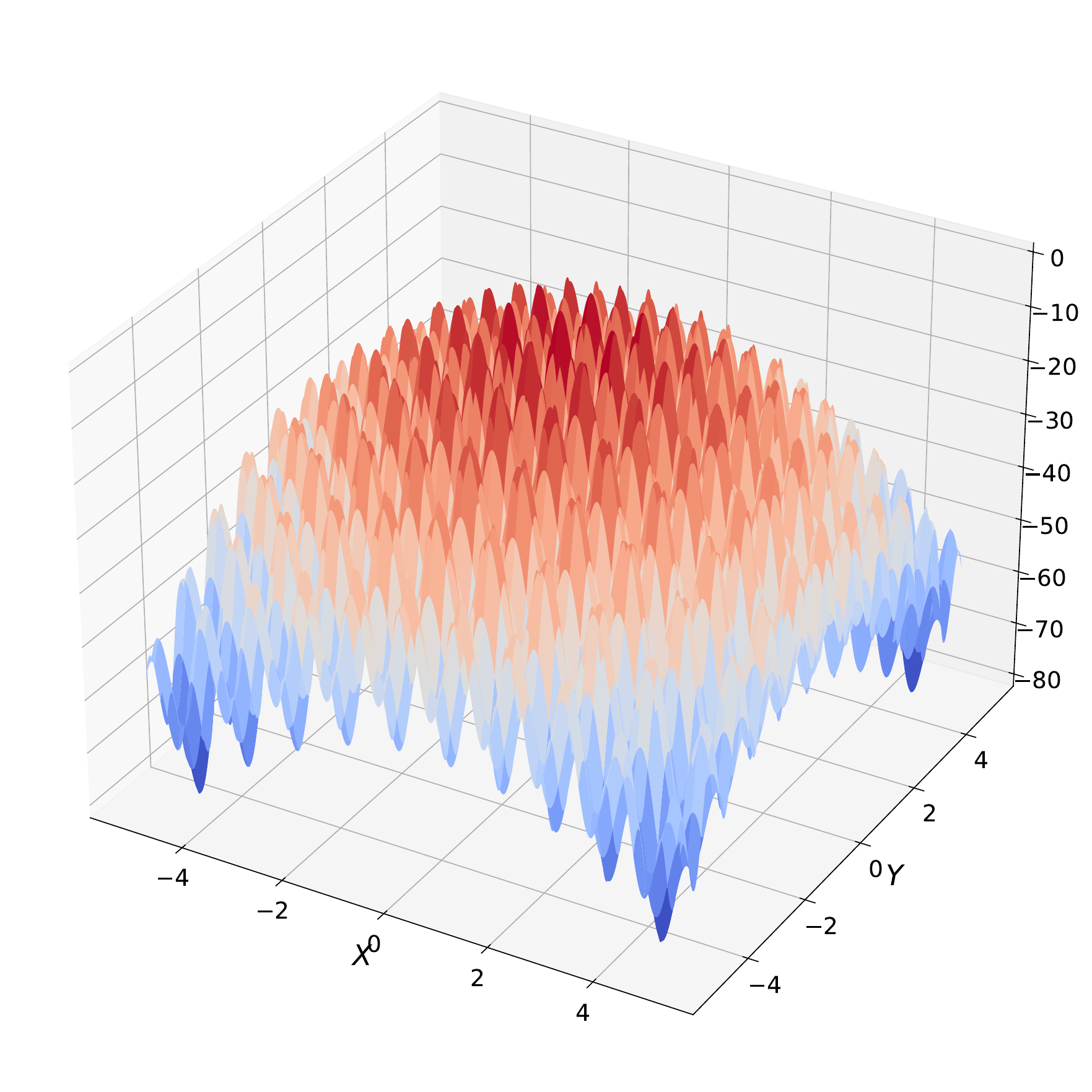}
      \caption{Rastrigin}
  \end{subfigure}
  \\
  \begin{subfigure}{0.25\textwidth}
      \centering
      \includegraphics[width=\textwidth]{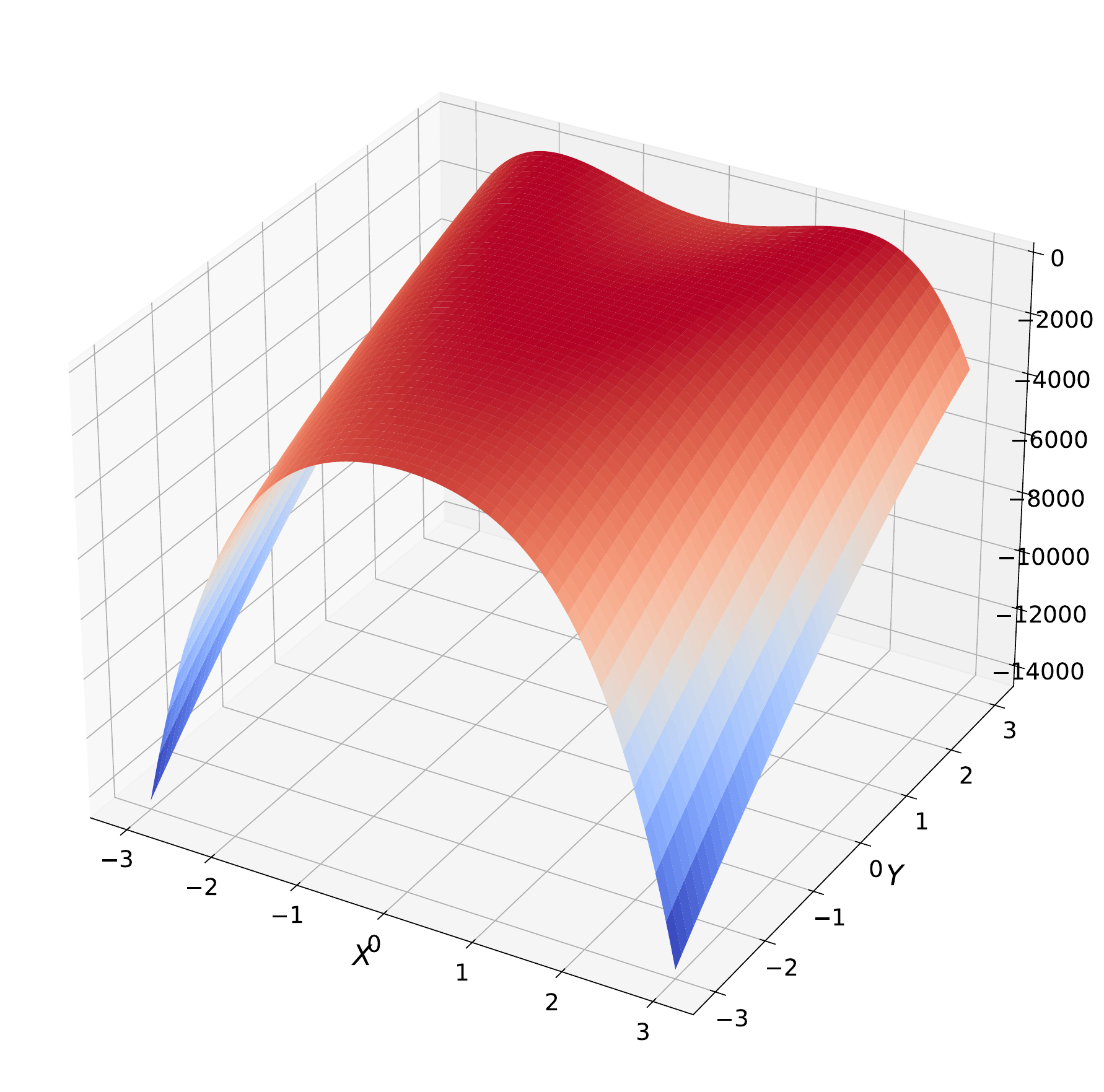}
      \caption{Rosenbrock}
  \end{subfigure}
	\hspace{1em}
  \begin{subfigure}{0.25\textwidth}
      \centering
      \includegraphics[width=\textwidth]{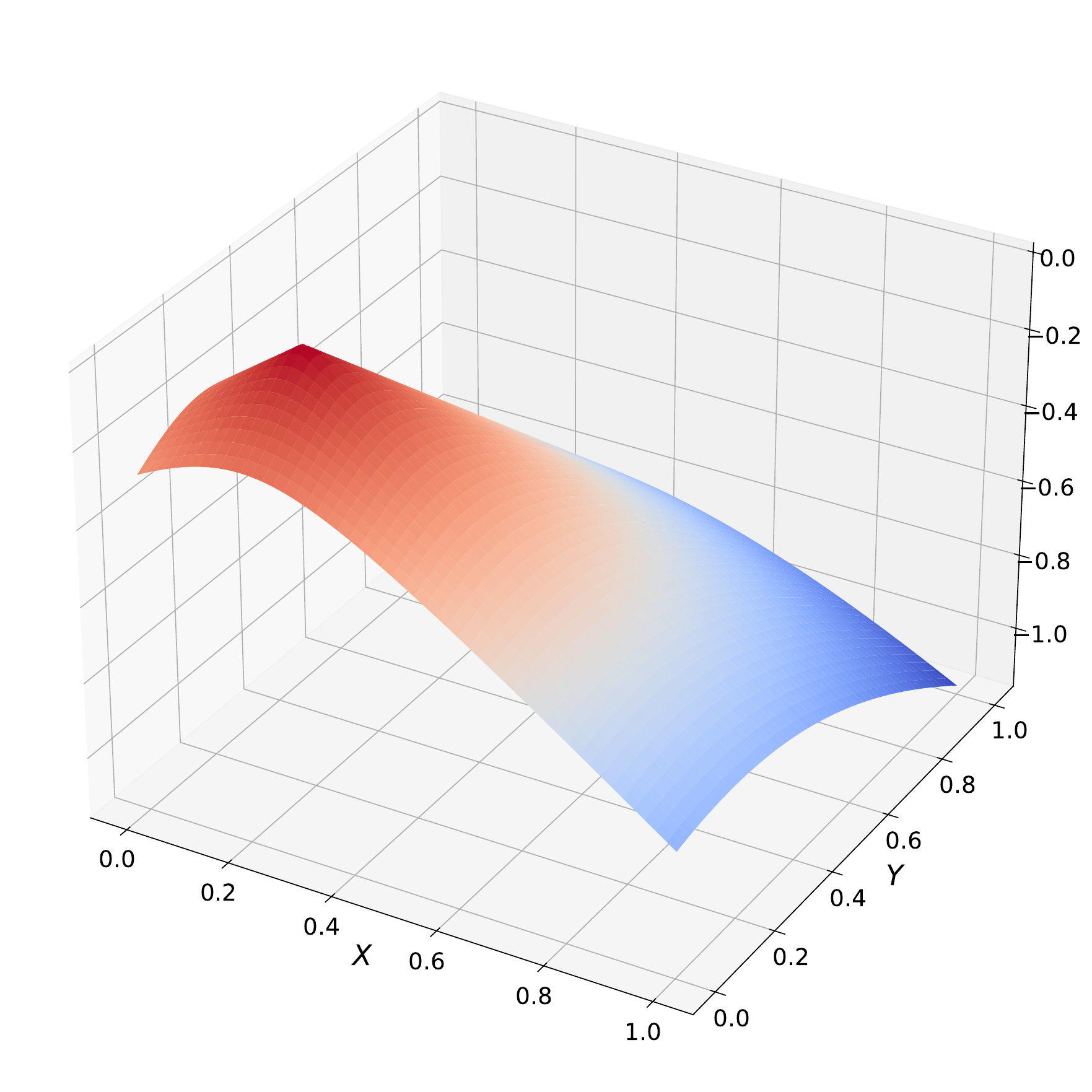}
      \caption{Sphere}
  \end{subfigure}
	\hspace{1em}
  \begin{subfigure}{0.25\textwidth}
      \centering
      \includegraphics[width=\textwidth]{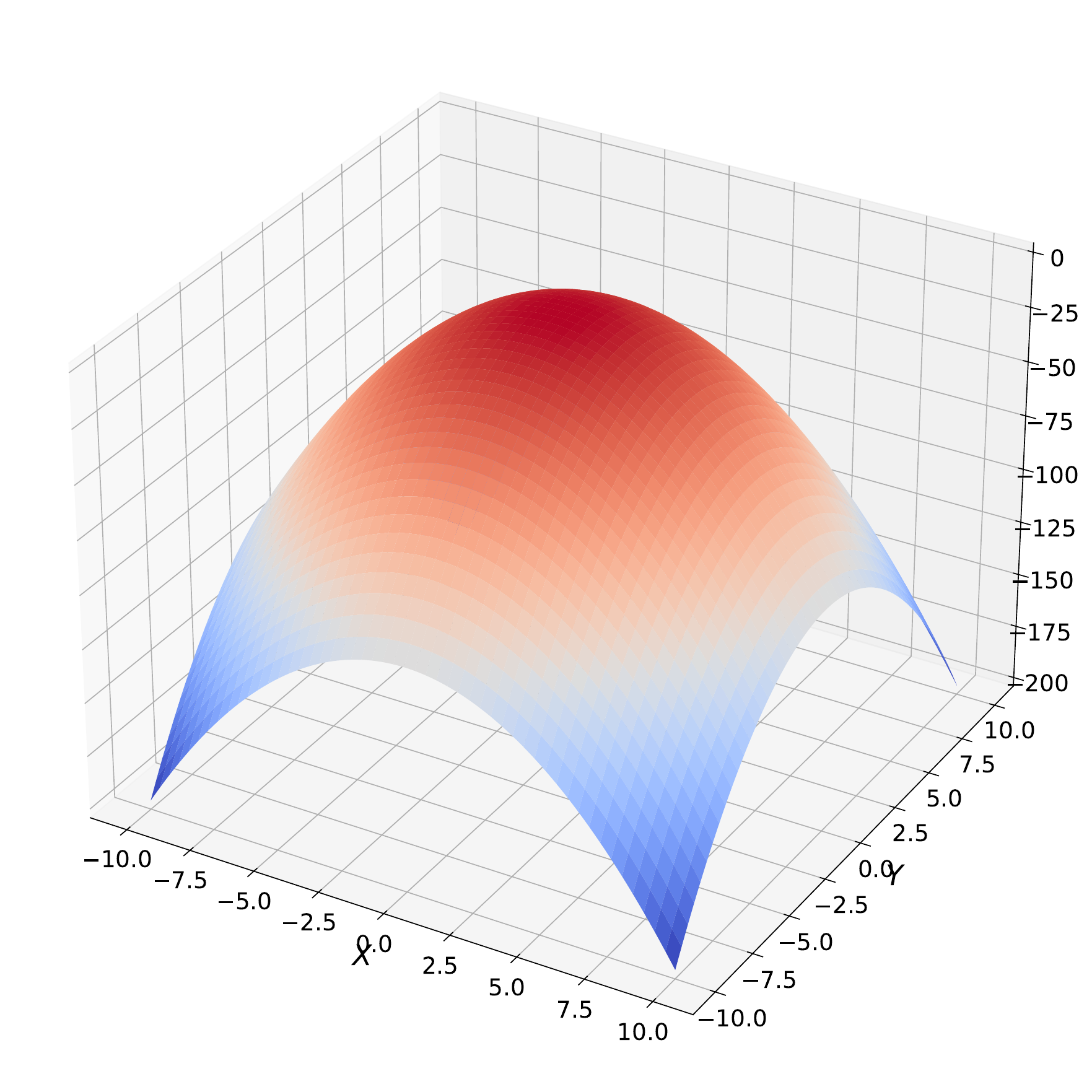}
      \caption{Square}
  \end{subfigure}
  \caption{Graphs of the chosen benchmark functions in 2D.}
	\label{fig:benchmark-functions}
\end{figure*}

\section{Experiments}
In this section, we compare the vanilla \lipo algorithms 
to ours, \lipop and \adalipop. 
We set $\Delta = 600$ and $w = 5$. We use standard benchmark functions used in global optimization literature \cite{simulationlib}. Some have few local minima (\eg the Sphere function) while others have many (\eg the Rastrigin function). See \cref{fig:benchmark-functions} for a visual representation of the functions. We set the dimension to $d = 2$ for all functions. We run each algorithm 100 times on each function and 
report the results in \cref{tab:exp}. The budget (maximum number of evaluations) depends on the function as some are easier to optimize than others. \lipo and \adalipo exhaust the entire budget, while \lipop and \adalipop 
may terminate earlier if the stopping criterion is met. One can see that, with significantly less evaluations, \lipop and \adalipop compete with the original version. 

In addition, we consider the \adalipob variant of \adalipop without using the stopping criterion defined in \cref{sec:stopping-criterion}. Same as the algorithms of the \lipo family, \adalipob stops when the budget is exhausted. To compare the performance of \adalipob to the original algorithms, we give them an infinite budget and stop them whenever the following condition is met: $g(\theta) \leq \max\limits_{1 \leq i < t} f(X_i)$, where $\theta \in [0, 1]$ a chosen threshold and
\begin{equation}\label{eq:cond_stop}
  g(\theta) = \max_{x \in \X} f(x) - \left( \max_{x \in \X} f(x) - \int_\X \frac{f(x)}{\lambda(\X)} \mathrm{d}x \right) (1 - \theta),
\end{equation}
where $\lambda$ is the standard Lebesgue measure that generalizes the notion of volume of any open set. This condition allows us to stop the algorithm whenever we consider it has reached close enough to the true maximum. The distance required is controlled by $\theta$: the closer to $1$, the smaller the distance. We set $\theta = 0.99$ for all the functions. The results are recorded in \cref{tab:exp2}. As one can see, \adalipob significantly outperforms the original on this benchmark. It even succeeds to beat \lipo on almost every problem, while knowing less information on the function. It corroborates the fact our decaying exploration rate is a good strategy to improve the empirical performance of \adalipo.

\begin{figure}[tb]
  \centering
  \begin{subfigure}{0.49\columnwidth}
    \centering
    \!\!\!\!\!\includegraphics[width=\textwidth]{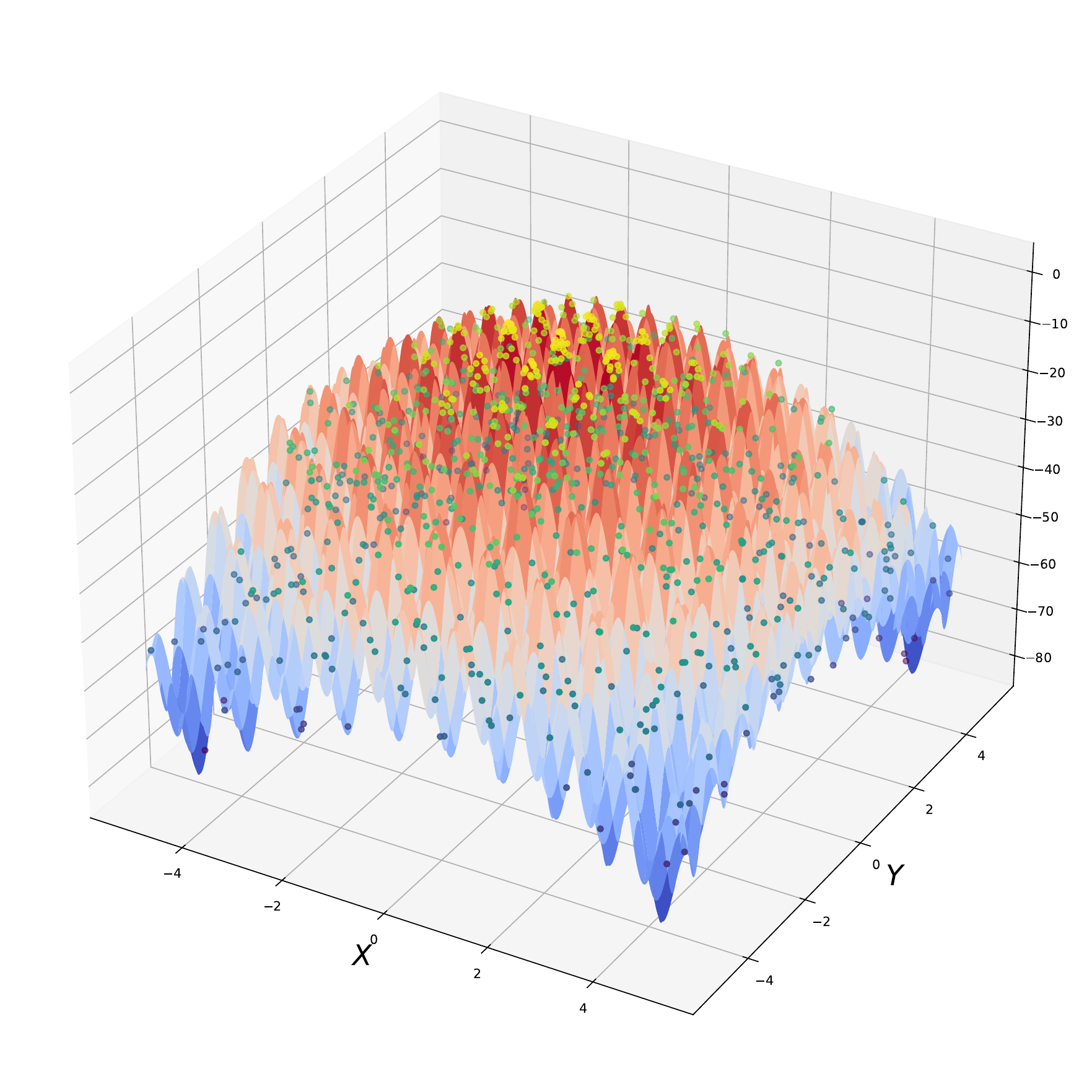}
    \caption{\adalipo}
  \end{subfigure}
  \begin{subfigure}{0.49\columnwidth}
    \centering
    \ \ \includegraphics[width=\textwidth]{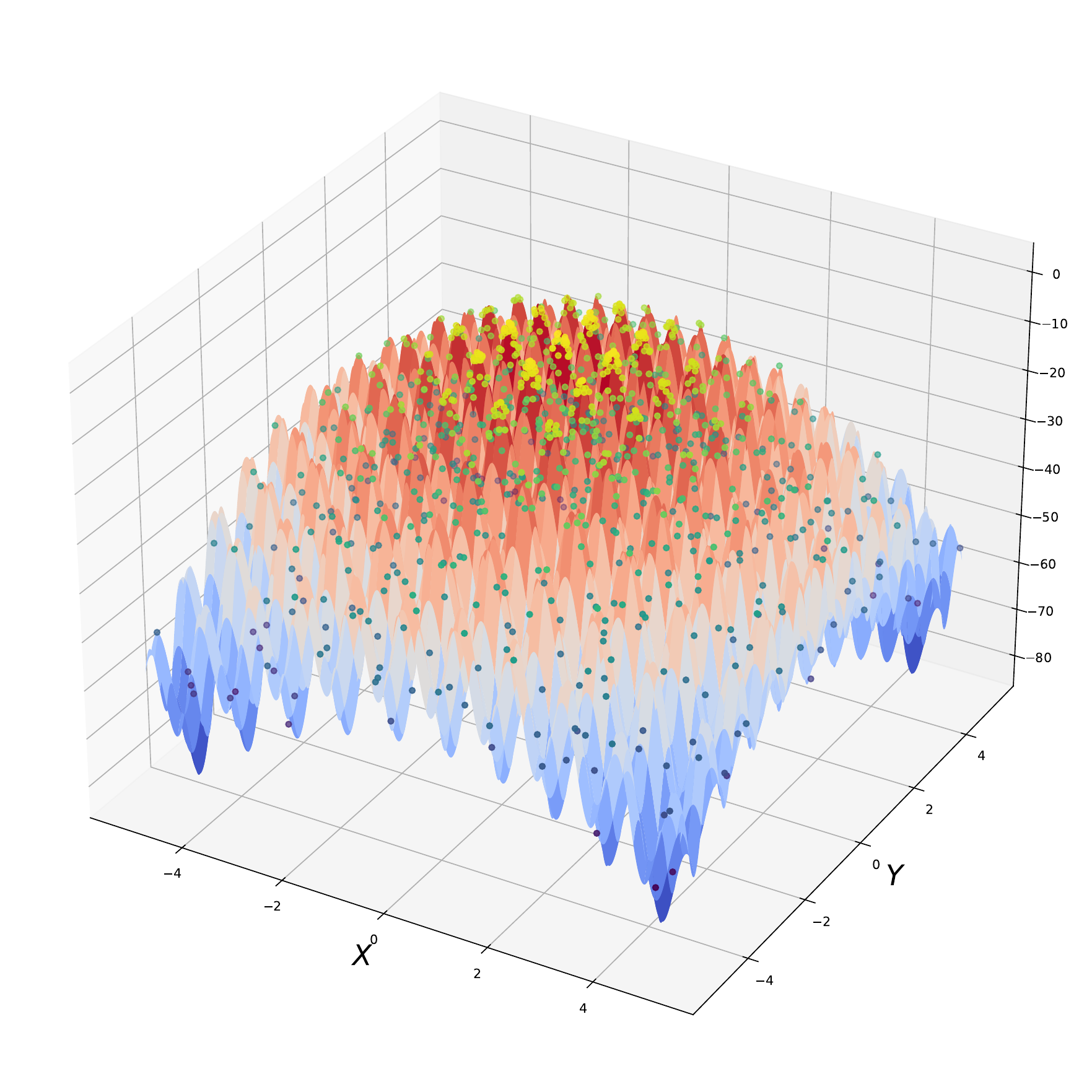}
    \caption{\adalipop}
  \end{subfigure}
  \caption{Visual understanding of the proposed improvements for \adalipo on the Rastrigin function, which has several local minima. The color shades represent the function value, from {\color[RGB]{58, 76, 192} blue} (low) to {\color[RGB]{179, 3, 38} red} or \textcolor{darkyellow}{yellow} (high). Comparing to \adalipo, the improved \adalipop not only it tends to be more restricted over the region of higher interest (which can be seen with a higher number of evaluations at the center in \textcolor{darkyellow}{yellow}), but it also reduces the number of function evaluations (dots).}
  \label{fig:adalipoe-visual}
\end{figure}

\renewcommand{\arraystretch}{1}
\setlength{\tabcolsep}{4pt} 
\begin{table}[t]
  \centering
  \fontsize{6.5pt}{6.5pt}\selectfont
  \caption{Empirical results comparing the original algorithms and the proposed improved versions. \#\,evals is the number of function evaluations (mean $\pm$ std), and $d_{\text{max}}$ is the distance from the real maximum. One can see that, with significantly less evaluations, \lipop and \adalipop compete with the original version in terms of distance to the real maximum.}%
  \begin{tabular}{c|cc|cc|cc}
       & \multicolumn{2}{c|}{Hölder} & \multicolumn{2}{c|}{Rastrigin} & \multicolumn{2}{c}{Sphere}\\
       & \#\,evals & $d_{\text{max}}$ & \#\,evals & $d_{\text{max}}$ & \#\,evals & $d_{\text{max}}$ \\
      \Xhline{2\arrayrulewidth} & & & & & & \\[-1.5ex]
      \lipo & 2000 & {\bt 0.0018} & 1000 & {\bt 0.0512} & 25 & 0.0306\\
      \lipop & 1505 $\pm$ 104 & {\bt 0.0018} & 869 $\pm$ 34 & 0.1282 & 25 & 0.0320\\
      \adalipo & 2000 & 0.003 & 1000 & 0.4106 & 25 & 0.0227\\
      \adalipop & {\bt 719 $\pm$ 457} & 0.023 & {\bt 753 $\pm$ 133} & 0.0569 & {\bt 20 $\pm$ 5} & {\bt 0.0063}\\
      \hline
  \end{tabular}
  \label{tab:exp}
\end{table}
\begin{table}[t]
  \centering
  \caption{Empirical performance of our 
	\adalipob variant. The table shows the total number of function evaluations (mean $\pm$ std) required to meet the condition stated in \cref{eq:cond_stop}. 
	\adalipob outperforms 
	\lipo and \adalipo in almost 
	all problems.}%
  \begin{tabular}{c|c|c|c}
       &\lipo & \adalipo & \adalipob \\
       \Xhline{2\arrayrulewidth} & & & \\[-1.5ex]
       Himmelblau & 100 $\pm$ 86 & 97 $\pm$ 77 & {\bt 65 $\pm$ 46} \\
       Hölder & 508 $\pm$ 217 & 319 $\pm$ 201 & {\bt 228 $\pm$ 136} \\
       Rastrigin & 670 $\pm$ 183 & 913 $\pm$ 297 & {\bt 616 $\pm$ 187} \\
       Rosenbrock & {\bt 11 $\pm$ 10} & 12 $\pm$ 11 & {\bt 11 $\pm$ 10} \\
       Sphere & 46 $\pm$ 10 & 28 $\pm$ 8 & {\bt 22 $\pm$ 6} \\
       Square & {\bt 43 $\pm$ 22} & 62 $\pm$ 47 & 51 $\pm$ 36 \\
       \hline
  \end{tabular}
  \label{tab:exp2}
\end{table}

\section{Limitations}
A limitation of \lipop and \adalipop is that the original proofs of consistency provided in \cite{Malherbe2017} concerning the vanilla versions are not directly applicable. We do not provide any theoretical guarantees on the convergence of our methods. Since the consistency of the original algorithms is one of their key features of the \lipo family, this is a major drawback for the proposed improvements. Another limitation related to the curse of dimensionality is inherited from the vanilla approach. As the dimension increases, the volume of potential maximizers increases exponentially and thus, the probability of rejecting a point sampled uniformly in the domain decreases accordingly. We provide the following upper bound for \lipo, which also holds for the other algorithms:
\begin{theorem}[\lipo rejecting probability]\label{thm:lipo-reject-prob}
  For any $\kappa$-Lipschitz function $f$, let $(X_i)_{1 \leq i \leq t}$ be the previous potential maximizers of \lipo at time $t$. For any $x \in \X$, let $R(x, t)$ be the event of rejecting $x$ at time $t + 1$, \ie
  $$
  R(x, t+1) = \max_{1 \leq i \leq t} f(X_i) < \min_{1 \leq i \leq t} f(X_i) + \kappa \lVert x - X_i \rVert_2.
  $$
  We have the following upper bound:
  $$
  \P(R(x, t+1)) \leq \frac{t \pi^{d / 2}\Delta^d}{\kappa^d \Gamma(d/2 + 1) \lambda(\X)},
  $$
  where $\Delta = \max_{x \in \X} f(x) - \min_{x \in \X} f(x) := \textup{diam}(\X)$ is the diameter of the domain, $\lambda$ is the standard Lebesgue measure that generalizes the notion of volume of any open set, and $\Gamma$ is the extended factorial function (\ie the Gamma function) given by $\Gamma(x) = \int_0^\infty t^{x - 1} e^{-t} \mathrm{d} t$.
\end{theorem}
\begin{proof}
  At time $t$, a candidate $x \in \X$ is rejected iff it belongs to a ball within $\X$:
  $$
  \min_{1 \leq i < t} f(X_i) + \kappa \lVert x - X_i \rVert_2 < \max_{1 \leq i < t} f(X_i).
  $$
  Let $j$ be in the $\arg\min$ of the LHS of the above inequality. It is equivalent to
  \begin{align*}
    &f(X_j) + \kappa \lVert x - X_j \rVert_2 < \max_{1 \leq i < t} f(X_i) \\
    \iff & \kappa \lVert x - X_j \rVert < \max_{1 \leq i < t} f(X_i) - f(X_j) \\
    \iff & x \in B \left(X_j, \frac{\max_{1 \leq i < t} f(X_i) - f(X_j)}{\kappa} \right) \bigcap \X \\
    & \;\;\; \subseteq B \left(X_j, \frac{\max_{1 \leq i < t} f(X_i) - f(X_j)}{\kappa} \right).
  \end{align*}
  As $\text{diam}(f(\X)) = \Delta$, the volume of a ball of radius $\frac{\Delta}{\kappa}$ is an upper bound on the volume that can be removed from the region of potential maximizers, for any sequence of iterations $(X_i)_{1 \leq i < t}$. Thus, at time $t+1$, at most the volume of $t$ disjoint balls of radius $\frac{\Delta}{\kappa}$ may have been removed. This leads to the following lower bound on the volume in which potential maximizers should be seek:
  $$
  v_{t+1} \geq \lambda(\X) - \frac{t \pi^{d / 2}\Delta^d}{\kappa^d \Gamma(d/2 + 1)}.
  $$
  As \lipo samples candidates uniformly at random in $\X$, the probability of rejecting a candidate is bounded from above by the probability of sampling uniformly at a point in the union of the $t$ disjoint balls:
  $$
  \P(R(x, t + 1)) \leq \frac{t \pi^{d / 2}\Delta^d}{\kappa^d \Gamma(d/2 + 1) \lambda(\X)}.
  $$
  We provide a formalization of this proof in Lean \cite{Moura2021} and its mathematical library Mathlib \cite{Mathlib2020} in Appendix\,\ref{app:Lean-proof}.
\end{proof}

This upper bound tends extremely quickly to $0$ as the dimension increases. For instance, let us consider the function $x \mapsto e^{\lVert x \rVert_2}$ over $[-1, 1]^d$, let $C_d = \frac{\pi^{d / 2}\Delta^d}{\kappa^d \Gamma(d/2 + 1) \lambda(\X)}$. Then, $C_2 = 0.78$, $C_5 = 0.16$, $C_{10} = 0.002$, $C_{50} = 1.5 \times 10^{-28}$. This implies that \lipo and, by extension, the other algorithms, tend to Pure Random Search as the dimension increases, since they accept any candidate with a high probability.

\section{Conclusion}
In this paper, we proposed simple yet effective empirical improvements to the algorithms of the \lipo family, which we respectively call name \lipop and \adalipop. We showed experimentally that our methods converge significantly faster than the vanilla versions, and hence they are more suitable for frugal optimization problems over Lipschitz functions. Our methods ship two major limitations: the lack of theoretical guarantees compared to the original algorithms, while it inherits from them the a limitation related to the curse of dimensionality. For the latter, we provided an upper bound on the probability of rejecting a candidate for \lipo, which tends very quickly to $0$ as the dimension increases.

\begin{acks}
This work was supported by the Industrial Data Analytics and Machine Learning Chair hosted at ENS Paris-Saclay.
\end{acks}

\bibliographystyle{ACM-Reference-Format}
\bibliography{refs}


\begin{thebibliography}{17}


\ifx \showCODEN    \undefined \def \showCODEN     #1{\unskip}     \fi
\ifx \showDOI      \undefined \def \showDOI       #1{#1}\fi
\ifx \showISBNx    \undefined \def \showISBNx     #1{\unskip}     \fi
\ifx \showISBNxiii \undefined \def \showISBNxiii  #1{\unskip}     \fi
\ifx \showISSN     \undefined \def \showISSN      #1{\unskip}     \fi
\ifx \showLCCN     \undefined \def \showLCCN      #1{\unskip}     \fi
\ifx \shownote     \undefined \def \shownote      #1{#1}          \fi
\ifx \showarticletitle \undefined \def \showarticletitle #1{#1}   \fi
\ifx \showURL      \undefined \def \showURL       {\relax}        \fi
\providecommand\bibfield[2]{#2}
\providecommand\bibinfo[2]{#2}
\providecommand\natexlab[1]{#1}
\providecommand\showeprint[2][]{arXiv:#2}

\bibitem[Davis et~al\mbox{.}(2022)]%
        {Davis2022}
\bibfield{author}{\bibinfo{person}{Damek Davis}, \bibinfo{person}{Dmitriy Drusvyatskiy}, \bibinfo{person}{Yin~Tat Lee}, \bibinfo{person}{Swati Padmanabhan}, {and} \bibinfo{person}{Guanghao Ye}.} \bibinfo{year}{2022}\natexlab{}.
\newblock \showarticletitle{A gradient sampling method with complexity guarantees for Lipschitz functions in high and low dimensions}. In \bibinfo{booktitle}{\emph{Proceedings of Advances in Neural Information Processing Systems}}.
\newblock


\bibitem[de~Moura and Ullrich(2021)]%
        {Moura2021}
\bibfield{author}{\bibinfo{person}{Leonardo de Moura} {and} \bibinfo{person}{Sebastian Ullrich}.} \bibinfo{year}{2021}\natexlab{}.
\newblock \showarticletitle{The {L}ean 4 Theorem Prover and Programming Language}.
\newblock In \bibinfo{booktitle}{\emph{Automated Deduction {\textendash} {CADE} 28}}. \bibinfo{publisher}{Springer International Publishing}.
\newblock


\bibitem[Gaëtan et~al\mbox{.}(2024)]%
        {Serre2024_IPOL}
\bibfield{author}{\bibinfo{person}{Serré Gaëtan}, \bibinfo{person}{Beja-Battais Perceval}, \bibinfo{person}{Chirrane Sophia}, \bibinfo{person}{Kalogeratos Argyris}, {and} \bibinfo{person}{Vayatis Nicolas}.} \bibinfo{year}{2024}\natexlab{}.
\newblock \bibinfo{title}{Implementation of Global Optimization Algorithms for Lipschitz Functions}.
\newblock
\newblock


\bibitem[Hansen and Ostermeier(1996)]%
        {Hansen1996}
\bibfield{author}{\bibinfo{person}{Nikolaus Hansen} {and} \bibinfo{person}{Andreas Ostermeier}.} \bibinfo{year}{1996}\natexlab{}.
\newblock \showarticletitle{Adapting arbitrary normal mutation distributions in evolution strategies: the covariance matrix adaptation}. In \bibinfo{booktitle}{\emph{Proceedings of the IEEE International Conference on Evolutionary Computation}}.
\newblock


\bibitem[Jordan et~al\mbox{.}(2023)]%
        {Jordan2023}
\bibfield{author}{\bibinfo{person}{Michael~I. Jordan}, \bibinfo{person}{Guy Kornowski}, \bibinfo{person}{Tianyi Lin}, \bibinfo{person}{Ohad Shamir}, {and} \bibinfo{person}{Manolis Zampetakis}.} \bibinfo{year}{2023}\natexlab{}.
\newblock \bibinfo{title}{Deterministic {Nonsmooth} {Nonconvex} {Optimization}}.
\newblock
\newblock


\bibitem[Kirkpatrick et~al\mbox{.}(1983)]%
        {Kirkpatrick1983}
\bibfield{author}{\bibinfo{person}{Scott Kirkpatrick}, \bibinfo{person}{C~Daniel Gelatt~Jr}, {and} \bibinfo{person}{Mario~P Vecchi}.} \bibinfo{year}{1983}\natexlab{}.
\newblock \showarticletitle{Optimization by simulated annealing}.
\newblock \bibinfo{journal}{\emph{science}} (\bibinfo{year}{1983}).
\newblock


\bibitem[Lee et~al\mbox{.}(2017)]%
        {Lee2017}
\bibfield{author}{\bibinfo{person}{Juyong Lee}, \bibinfo{person}{In-Ho Lee}, \bibinfo{person}{InSuk Joung}, \bibinfo{person}{Jooyoung Lee}, {and} \bibinfo{person}{Bernard~R. Brooks}.} \bibinfo{year}{2017}\natexlab{}.
\newblock \showarticletitle{Finding multiple reaction pathways via global optimization of action}.
\newblock \bibinfo{journal}{\emph{Nature Communications}} (\bibinfo{year}{2017}).
\newblock


\bibitem[Malherbe and Vayatis(2017)]%
        {Malherbe2017}
\bibfield{author}{\bibinfo{person}{C{\'e}dric Malherbe} {and} \bibinfo{person}{Nicolas Vayatis}.} \bibinfo{year}{2017}\natexlab{}.
\newblock \showarticletitle{Global optimization of {L}ipschitz functions}. In \bibinfo{booktitle}{\emph{Proceedings of the International Conference on Machine Learning}}.
\newblock


\bibitem[Martinez-Cantin(2014)]%
        {Martinezcantin2014}
\bibfield{author}{\bibinfo{person}{Ruben Martinez-Cantin}.} \bibinfo{year}{2014}\natexlab{}.
\newblock \showarticletitle{BayesOpt: A Bayesian Optimization Library for Nonlinear Optimization, Experimental Design and Bandits}.
\newblock \bibinfo{journal}{\emph{Journal of Machine Learning Research}} (\bibinfo{year}{2014}).
\newblock


\bibitem[mathlib Community(2020)]%
        {Mathlib2020}
\bibfield{author}{\bibinfo{person}{The mathlib Community}.} \bibinfo{year}{2020}\natexlab{}.
\newblock \showarticletitle{The {L}ean mathematical library}. In \bibinfo{booktitle}{\emph{Proceedings of the {ACM} {SIGPLAN} International Conference on Certified Programs and Proofs}}.
\newblock


\bibitem[Mirjalili and Lewis(2016)]%
        {Mirjalili2016}
\bibfield{author}{\bibinfo{person}{Seyedali Mirjalili} {and} \bibinfo{person}{Andrew Lewis}.} \bibinfo{year}{2016}\natexlab{}.
\newblock \showarticletitle{The whale optimization algorithm}.
\newblock \bibinfo{journal}{\emph{Advances in engineering software}} (\bibinfo{year}{2016}).
\newblock


\bibitem[Pint{\'e}r(1991)]%
        {Pinter1991}
\bibfield{author}{\bibinfo{person}{J{\'a}nos~D Pint{\'e}r}.} \bibinfo{year}{1991}\natexlab{}.
\newblock \showarticletitle{Global optimization in action}.
\newblock \bibinfo{journal}{\emph{Scientific American}} (\bibinfo{year}{1991}).
\newblock


\bibitem[Rudi et~al\mbox{.}(2024)]%
        {Rudi2024}
\bibfield{author}{\bibinfo{person}{Alessandro Rudi}, \bibinfo{person}{Ulysse Marteau-Ferey}, {and} \bibinfo{person}{Francis Bach}.} \bibinfo{year}{2024}\natexlab{}.
\newblock \bibinfo{title}{Finding global minima via kernel approximations}.
\newblock
\newblock


\bibitem[Serré et~al\mbox{.}(2024)]%
        {Serre2024}
\bibfield{author}{\bibinfo{person}{Gaëtan Serré}, \bibinfo{person}{Argyris Kalogeratos}, {and} \bibinfo{person}{Nicolas Vayatis}.} \bibinfo{year}{2024}\natexlab{}.
\newblock \bibinfo{title}{Stein Boltzmann Sampling: A Variational Approach for Global Optimization}.
\newblock
\newblock
\showeprint[arxiv]{2402.04689}


\bibitem[Surjanovic and Bingham(2022)]%
        {simulationlib}
\bibfield{author}{\bibinfo{person}{S. Surjanovic} {and} \bibinfo{person}{D. Bingham}.} \bibinfo{year}{2022}\natexlab{}.
\newblock \bibinfo{title}{Virtual Library of Simulation Experiments: Test Functions and Datasets}.
\newblock \bibinfo{howpublished}{Retrieved from: \url{http://www.sfu.ca/~ssurjano}}.
\newblock


\bibitem[Xue and Shen(2023)]%
        {Xue2023}
\bibfield{author}{\bibinfo{person}{Jiankai Xue} {and} \bibinfo{person}{Bo Shen}.} \bibinfo{year}{2023}\natexlab{}.
\newblock \bibinfo{title}{Dung beetle optimizer: A new meta-heuristic algorithm for global optimization}.
\newblock
\newblock


\bibitem[Zhang et~al\mbox{.}(2020)]%
        {Zhang2020}
\bibfield{author}{\bibinfo{person}{Jingzhao Zhang}, \bibinfo{person}{Hongzhou Lin}, \bibinfo{person}{Stefanie Jegelka}, \bibinfo{person}{Suvrit Sra}, {and} \bibinfo{person}{Ali Jadbabaie}.} \bibinfo{year}{2020}\natexlab{}.
\newblock \showarticletitle{Complexity of Finding Stationary Points of Nonconvex Nonsmooth Functions}. In \bibinfo{booktitle}{\emph{Proceedings of the International Conference on Machine Learning}}.
\newblock


\end{thebibliography}

\appendix

\section{Formalization of Theorem 5.1}\label{app:Lean-proof}
Lean~\cite{Moura2021} is a programming language that facilitates the writing of mathematical proofs. The typechecker of Lean ensures that the proof is correct. The mathematical library Mathlib~\cite{Mathlib2020} provides a wide range of mathematical tools and theorems that can be used to prove other mathematical statements. We provide a formalization of the proof of \cref{thm:lipo-reject-prob} in Lean using Mathlib. In this section, we only provide the definition of the main objects and the statement of the main theorems. The complete code is available online\footnote{\scriptsize Source code:~\url{https://github.com/gaetanserre/Lean-LIPO}.}.

\subsection{Definitions}
We first define the dimension of the space $d$, such that $0 < d$.
\begin{figure}[H]
  \centering
  \includegraphics{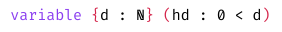}
\end{figure}
We also need to define our search space $\X \subset \R^d$. Note that it is not required to be compact, but only to be approximated by a measurable set, up to a null measure set.
\begin{figure}[H]
  \centering
  \includegraphics{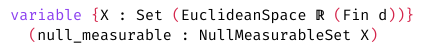}
\end{figure}
This allows us to define the uniform measure over our space.
\begin{figure}[H]
  \centering
  \includegraphics{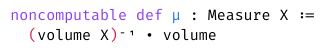}
\end{figure}
Then, we define the function to be optimized $f : \X \to \R$ where the sets of $\arg\max$ and $\arg\min$ of $f$ are supposed to be non-empty.
\begin{figure}[H]
  \centering
  \includegraphics{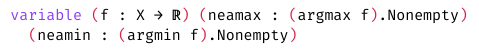}
\end{figure}
Note that this is a slightly more general framework than the one presented in the paper, where $f$ is continuous and defined over a compact (which implies the properties of the above definition).

Next, for a given set of potential maximizers $A$ and a Lipschitz constant $\kappa$, we define the event ``\emph{a candidate $x$ is being rejected by \lipo\!\!}".
\begin{figure}[H]
  \centering
  \includegraphics{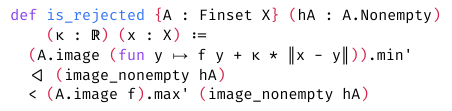}
\end{figure}
This allows us to define the set of all rejected candidates, given a set of potential maximizers $A$ and a Lipschitz constant $\kappa$.
\begin{figure}[H]
  \centering
  \includegraphics{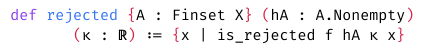}
\end{figure}
We define the diameter of the image of $f$ as $f(x) - f(y)$, for any $x \in \arg\max_{x \in A} f(x)$ and $y \in \arg\min_{x \in A} f(x)$.
\begin{figure}[H]
  \centering
  \includegraphics{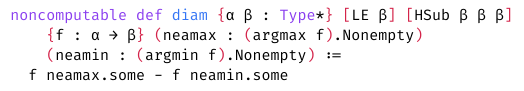}
\end{figure}
Finally, we define the volume of a ball of radius $\frac{\text{diam}}{\kappa}$.
\begin{figure}[H]
  \centering
  \includegraphics{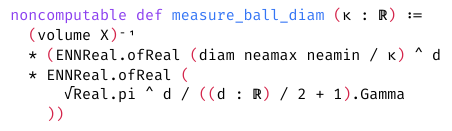}
\end{figure}

\subsection{Theorems}
We prove that, given a set of potential maximizers $A$ and a Lipschitz constant $\kappa$, a candidate $x$ is rejected by \lipo iff there exists a point $x'$ in $A$ such that $x \in B\left(x', \frac{\max_{y \in A} f(y) - f(x')}{\kappa}\right)$.
\begin{figure}[H]
  \centering
  \includegraphics{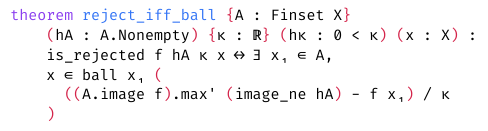}
\end{figure}
This allows us to prove that the set of all rejected candidates is equal to the union indexed by $A$ of balls defined as above.
\begin{figure}[H]
  \centering
  \includegraphics{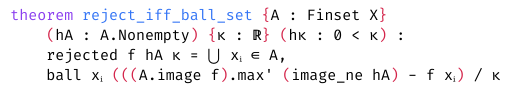}
\end{figure}
Finally, using classical results on restricted measure, on the volume of balls in Euclidean space, and the fact that the diameter is bigger than any distance between two points in the image of $f$, we can prove \cref{thm:lipo-reject-prob}.
\begin{figure}[H]
  \centering
  \includegraphics{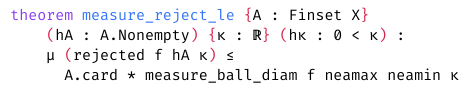}
\end{figure}

\end{document}